%% file: no-d-on-Q-quadric.tex
\documentclass[a4paper,12pt]{article}

\usepackage[british]{babel}
\usepackage{csquotes}
\usepackage[margin=2cm]{geometry}
\usepackage{amssymb,amsmath,amsthm,thmtools,mathtools}
\usepackage[dvipsnames]{xcolor}
\usepackage[hidelinks]{hyperref}
\addto\extrasbritish{%
	\def\Snospace~{\S{}}
}
\usepackage{enumitem}
\usepackage[backend=biber,style=alphabetic,sorting=nyt,backref=true,useprefix=true,maxalphanames=4,minalphanames=3,maxnames=4,minnames=3]{biblatex}
\newenvironment{summary}{\begin{quote}\small}{\end{quote}}

\declaretheorem[style=plain,name=Theorem,numberwithin=section]{theorem}
\declaretheorem[style=plain,name=Proposition,sibling=theorem]{proposition}
\declaretheorem[style=plain,name=Lemma,sibling=theorem]{lemma}
\declaretheorem[style=plain,name=Problem,sibling=theorem]{problem}
\declaretheorem[style=plain,name=Corollary,sibling=theorem]{corollary}

\declaretheorem[style=definition,name=Definition,sibling=theorem]{definition}

\declaretheorem[style=remark,name=Remark,sibling=theorem]{remark}

\renewcommand{\phi}{\varphi}
\renewcommand{\rho}{\varrho}
\renewcommand{\epsilon}{\varepsilon}

\newcommand{\field}{\mathbb{F}}
\newcommand{\R}{\mathbb{R}}
\newcommand{\Q}{\mathbb{Q}}
\newcommand{\Z}{\mathbb{Z}}
\newcommand{\N}{\mathbb{N}}
\newcommand{\K}{\mathbb{K}}
\newcommand{\leteq}{\coloneq}
\newcommand{\from}{\colon}

\DeclareMathOperator{\Span}{span}
\DeclareMathOperator{\im}{im}
\DeclareMathOperator{\PGL}{PGL}

\newcommand{\Variety}[2]{\mathrm{V}(#2|\,#1)}
\newcommand{\Proj}[2]{\mathbb{P}_{#1}^{#2}}
\newcommand{\Aff}[2]{\mathbb{A}_{#1}^{#2}}
\newcommand{\Veronese}[1]{\nu_{#1}}
\newcommand{\ProjClosure}[1]{\overline{#1}}
\newcommand{\AlgClosure}[1]{\bar{#1}}

\newcommand{\Legendre}[2]{\genfrac{(}{)}{}{}{#1}{#2}}

\newcommand{\noThree}{\textsf{no-three-in-line}}
\newcommand{\noK}[1]{\textsf{no-$(k+1)$-in-line}}
\newcommand{\noFour}{\textsf{no-four-on-a-circle}}
\newcommand{\noSphere}[1]{\textsf{no-$(#1+2)$-on-a-sphere}}
\newcommand{\noSharp}[1]{\textsf{sharp no-$(#1+2)$-on-a-sphere}}
\newcommand{\noQuadric}[2]{\textsf{no-$(#1+2)$-on-$#2$-quadric}}

\title{Rational normal curves as no-$(d+2)$-on-$Q$-quadric sets}
\author{Dávid R. Szabó\thanks{HUN-REN Alfréd Rényi Institute of Mathematics 
(Hungary, 1053 Budapest, Reáltanoda u. 13-15.) and   
ELTE Linear Hypergraphs  Research Group. 
The author is supported by the University Excellence Fund of Eötvös Loránd University. 
E-mail: {\tt szabo.r.david@gmail.com}}}
\date{\today{}}

\begin{document}

\maketitle

\begin{abstract}
    For every $d\geq 2$, 
    we construct a subset $D\subseteq \{1,2,\dots,n\}^d$ of size $n-o(n)$ such that 
    every affine hyperplane of $\R^d$ intersects $D$ in at most $d$ points, and 
    every hypersphere of $\R^n$ intersects $D$ in at most $d+1$ points. 
    This construction is the largest one currently known, 
    and strongly builds on ideas of Dong, Xu, and also of Thiele. 
    More generally, we prove that the role of hyperspheres can be replaced 
    by $Q$-quadrics, i.e. by quadratic surfaces given by an equation 
    whose degree two homogeneous part equals a fixed quadratic form $Q$. 
    We formulate analogous statements in affine spaces over (finite) fields. 
    Essentially, every construction is given by a suitable rational normal curve in a $d$-dimensional projective space.
\end{abstract}

\tableofcontents

\input{parts/1-introduction}

\input{parts/2-preliminaries}
\input{parts/3-Q-quadrics}

\input{parts/4-constructions}
\input{parts/5-open}

\printbibliography

\end{document}

%% file: parts/1-introduction.tex
\section{Introduction}

\subsection{History}

\paragraph{The \noThree{} problem} 
In \citeyear{Dudeney}, \citeauthor{Dudeney}  asked at most how many points can be selected from  an $n\times n$ grid of the Euclidean plane so that no three of these points lie on the same line \cite{Dudeney}. 
The so-called \noThree{} problem of determining this maximum value is still open in general (apart from small values of $n$ \cite{flammenkamp1998progress}) despite receiving considerable attention.  
The best known upper bound is still the obvious $2n$ obtained by considering the rows. 
The current best known lower bound of $\frac{3}{2}n-o(n)$ by \citeauthor{HJSW}   stands since \citeyear{HJSW} \cite{HJSW}. 
This construction uses degree $2$ algebraic curves over finite fields. 
Several conjectures were formulated concerning
 the asymptotic value \cite{Guy/Kelly:1968, Brass2005, Eppstein:2018,  green100}.

\paragraph{Variations on a theme}
Many variants of the \noThree{} problem have been studied. 
For example, for the \noK{k} problem (where at most $k$ points are allowed on every line), the obvious upper bound $kn$ was recently shown to sharp by \citeauthor{noKlarge} for $k>C\sqrt{n\log(n)}$ \cite{noKlarge} which was extended by \citeauthor{GrebennikovKwan} to $n\geq k\geq 10^{37}$ \cite{GrebennikovKwan}, while close to optimal constructions are known for smaller values of $k$ \cite{noKconst}. 
See for example \cite{brass2003counting} by \citeauthor{brass2003counting}, \cite{Lefmann} by \citeauthor{Lefmann}, \cite{por2007no} by \citeauthor{por2007no} or \cite{suk2025higherdimensionalpointsets} by \citeauthor{suk2025higherdimensionalpointsets} for a higher dimensional generalisation. 
\citeauthor{nagy2023extensible} considered the extensible version of the problem to $\Z^2$ \cite{nagy2023extensible}.  
For more  related general position problems, see the survey \cite{SurveyKlavvzar2025general} by \citeauthor{SurveyKlavvzar2025general}.

\paragraph{The \noFour{} problem}
Another related theme additionally involves quadratic hypersurfaces. Erdős and Purdy asked to determine the maximum number of points that can be selected from an $n\times n$ grid so that no $4$ of them are on a line or on a circle, known as the \noFour{} problem \cite[\S F6]{guy2004unsolved}. 
\citeauthor{ThielePhD} in \citeyear{ThielePhD} showed this number is between $\frac{1}{4}n$ and $\frac{5}{2}n-\frac{3}{2}$ \cite[\S1.1]{ThielePhD} or \cite{thiele1995no}. 
In his construction, actually, no line contains three points. 
\citeauthor{ThielePhD} \cite[\S1.2]{ThielePhD} and \citeauthor{Brass2005} \cite[\S10.1, Problem~4]{Brass2005} posed to determine the maximum number of points from 
$[n]\leteq \{1,\dots,n\}^d\subset \R^d$ such that no $d+2$ points are contained in a hyperplane or in a hypersphere, known as the \noSphere{d} problem.  
Considering parallel hyperplanes, we see that at most $(d+1)n$ points may be selected. 
The first non-trivial construction was given in \citeyear{ThielePhD}, \citeauthor{ThielePhD} \cite[\S1.2]{ThielePhD} gave an algebraic construction of size $\Omega(n^{1/(d-1)})$. 
Recently, \citeauthor{SukWhite2024} \cite[Theorem 1.1]{SukWhite2024} gave a construction based on VC-dimensions of size $n^{3/(d+1)-o(n)}$. 
In turn, this was improved further by \citeauthor{Keevash2025subsetslatticecubesavoiding} \cite[Corollary 1.6]{Keevash2025subsetslatticecubesavoiding} using randomised constructions to $\Omega(n^{(\min\{d,4\})/(d+1) - c/\log\log n})$. 
The best known construction of size $n-o(n)$ is due to \citeauthor{DongXu2025} \cite[Theorem 2]{DongXu2025} and uses algebraic curves.

\subsection{Main results}

\paragraph{The \noSharp{d} problem}
Following \cite[\S1.2]{ThielePhD}, in this paper, we consider the following slightly sharper (and maybe more natural) version of the previous problem.

\begin{problem}[The \noSharp{d} problem]
\label{p:noSharp}
    For $n\in\N$ and $d\geq 2$, 
    determine the size, 
    say $f([n]^d,\R^d)$, 
    of the largest subset $D\subseteq [n]^d\subset \R^d$  satisfying the \noSharp{d} condition in $\R^d$, i.e. 
    \begin{itemize}[nosep]
        \item no $d+1$ points of $D$ lie on a common affine hyperplane in $\R^d$, and 
        \item no $d+2$ points of $D$ lie on a common hypersphere in $\R^d$.
    \end{itemize}     
\end{problem}
By considering parallel hyperplanes, we see that $f([n]^d,\R^d)\leq d\cdot n$. 
The aforementioned construction of \citeauthor{ThielePhD} from \citeyear{ThielePhD} actually shows that $\frac{3}{32}n^{1/(d-1)}\leq f([n]^d,\R^d)$, see \cite[\S1.2]{ThielePhD} or \cite[footnote]{thiele1995no}. 
Recently, this was improved to $n/(d+1)-o(n)\leq f([n]^d,\R^d)$ by \citeauthor{DongXu2025} \cite[Theorem 2]{DongXu2025}. 
(Note that in \cite{DongXu2025}, the notation $\operatorname{ex}([n]^d;(d+1)_{\scriptscriptstyle\square},(d+2)_{\scriptscriptstyle\bigcirc})$ is used for our $f([n]^d,\R^d)$.)
Both constructions used  algebraic curves over finite fields in a similar way.
The following main result of this paper improves the previous best lower bound by a factor of $d+1$.
\begin{theorem}[The \noSharp{d} problem]
\label{thm:mainShpereGrid}
    For every $d\geq 2$, as $n\to\infty$, we have 
    \[n-o(n)\leq f([n]^d,\R^d)\leq d\cdot n.\]
\end{theorem}

\paragraph{The \noQuadric{d}{Q} problem} 
Note that the equation of every hypersphere in $\R^d$ is a quadratic polynomial $F\in \R[X_1,\dots,X_d]$ whose degree $2$ \emph{homogeneous} part is the fixed polynomial $Q=X_1^2+\dots+X_d^2$. 
Extending this, 
for an arbitrary \emph{quadratic form} (i.e. homogeneous degree $2$ polynomial) $Q\neq 0$, 
a surface is said to be a \emph{$Q$-quadric} (cf. \autoref{def:Qquadric}) if its equation is a degree $2$ polynomial whose homogeneous degree $2$ part is exactly $Q$ (up to scalar multiples).  
Similarly to hyperspheres, a $Q$-quadric is uniquely determined  by $d+1$ general position points (see \autoref{lem:uniqueQquadric}), which motivates the following problem.

\begin{problem}[The \noQuadric{d}{Q} problem]
\label{p:noQuadric}
    For $n\in\N$, $d\geq 2$ and quadratic form $0\neq Q\in\Q[X_1,\dots,X_d]$, 
    determine the size, 
    say $f_Q([n]^d,\R^d)$, 
    of the largest subset $D\subseteq [n]^d\subset \R^d$  satisfying the \emph{\noQuadric{d}{Q} condition in $\R^d$}, i.e. 
    \begin{itemize}[nosep]
        \item no $d+1$ points of $D$ lie on a common affine hyperplane in $\R^d$, and 
        \item no $d+2$ points of $D$ lie on a common $Q$-quadric in $\R^d$.
    \end{itemize}     
\end{problem}

The next statement is the other main result of this paper, which immediately implies the previous \autoref{thm:mainShpereGrid} as a special case, which we prove in \autoref{sec:ConstrGrid}.
\begin{theorem}[The \noQuadric{d}{Q} problem]
\label{thm:mainQgrid}
    For every $d\geq 2$ and ever quadratic form $0\neq Q\in  \Q[X_1,\dots,X_d]$, 
    we have 
    \[n-o(n)\leq f_Q([n]^d,\R^d)\leq d\cdot n.\]
\end{theorem}

\paragraph{Finite field analogues}

In turn, \autoref{thm:mainQgrid} is actually a consequence of an analogous statement over finite fields $\field$. 
In this setup, we pick the points from the full affine space $\field^d$ so that every affine hyperplane in $\field^d$ can contain at most $d$ selected points, and 
every $Q$-quadric in $\field^d$ can contain at most $d+1$ selected points. Such configurations are said to be \emph{$Q$-generic} (see \autoref{def:Qgeneric}), and its minimal size denoted by  $f_Q(\field^d,\field^d)$. 

\begin{theorem}[The \noQuadric{d}{Q} problem over finite fields]
\label{thm:QgenericoverFinite}
    Let $\field$ be a finite field and $2\leq d\in\N$ with $d\leq |\field|+1$. 
    Let $0\neq Q\in\field[X_1,\dots,X_d]$ be any quadratic form which is not irreducible of rank $2$.
    Then \[|\field|+1-d\leq f_Q(\field^d,\field^d)\leq d\cdot |\field|.\]
\end{theorem}
We prove this statement in \autoref{sec:ConstrFinite} by considering the affine part of a suitably constructed rational normal curve over $\field$. 
It is worth specialising the previous statement to the case of hyperspheres by taking $Q=X_1^2+\dots+X_d^2$.
\begin{theorem}[The \noSharp{d} problem over finite fields]
\label{thm:mainSphhereFiniteFields}
    Let $\field$ be a finite field and $d\in \N$ with $d\leq |\field|+1$. 
    If $d=2$, assume further that $|\field|\not\equiv 3\pmod{4}$. 
    Then there is a $C_0\subseteq \Aff{\field}{d}$ of size $|\field|+1-d$ 
    satisfying the \noSharp{d} condition in $\Aff{\field}{d}$.
\end{theorem}

\subsection{Idea of the construction}
Following most algebraic constructions, we will mostly work over a field of prime order. 
We mix the most useful ideas from \cite{ThielePhD} and \cite{DongXu2025}.

\paragraph{Motivation from \citeauthor{ThielePhD}: degree $d$ algebraic curves}
\cite{ThielePhD} considered the degree $d$ modular momentum curve which automatically solved  
every hyperplane. 
Since hyperspheres intersect this curve in too many points, he systematically deleted many points of his curve so that at most $d+1$ points remained on every hypersphere.

\paragraph{Motivation from \citeauthor{DongXu2025}: hiding intersection points at infinity} 
This idea is from \cite{DongXu2025} (but similar phenomena have appeared much less explicitly in \cite{HJSW}).  
By Bézout's theorem, every projective curve of degree $d$ (not contained in a hypersphere) intersects every hypersphere in exactly $2d+2$ points (counted with multiplicity, over the algebraic closure, and considering points of the ideal hyperplane).
\cite{DongXu2025} carefully constructed such a curve whose $d+1$ points of intersection with the ideal hyperplane (at infinity) actually are points of \emph{every} hypersphere, 
hence leaving at most $d+1$ intersection points in the affine space.
They constructed the parametrisation for such a curve by a relatively long computation, where they also had to ensure that the curve is not contained in any hyperspheres. 
The degree of the curve may produce too many intersection points with hyperplanes, so (similarly to \cite{ThielePhD}), \cite{DongXu2025} also removed many points of the curve so that every hyperplane contains at most $d$ points.

For example, in the case $d=2$, every circle intersects the ideal line in exactly the same two fixed points $P_1$ and $P_2$. The degree $3$ curve constructed in \cite{DongXu2025} passes through $P_i$ with multiplicity $i$. This tangential property at $P_2$ shows that the curve is not contained in any circle. 
The geometric picture in higher dimensions in \cite{DongXu2025} is analogous.

\paragraph{Our construction} 
We present a structural construction that is almost free of computations and works as well to arbitrary $Q$-quadrics, not just to hyperspheres.

We aim to choose the degree of our curve to be as small as possible without being contained in any hyperplane. 
It is a classical algebraic geometry fact (see \autoref{rem:RNCdegChar}) the smallest possible degree is $d$ (if the dimension of the ambient space is $d$) and that 
up to projective linear transformations, there is a \emph{unique} irreducible non-degenerate degree $d$ curve in the $d$-dimensional projective space, the so called \emph{rational normal curve}. (This characterisation is not used in our proof, but it helped to find the construction.)

We start with the moment curve (i.e. the Veronese embedding) of \cite{ThielePhD}, which, by above, is just a special instance of rational normal curves. 
Next, we apply a suitable (projective) linear transformation to it that enables keeping the whole (affine) curve in our construction without the need to delete additional points. 
This is the step where our construction beats in size the former ones.

To find this suitable transform, 
we pick points $P_1,\dots,P_d$ in general position from the ideal hyperplane such that 
$P_1,\dots,P_{d-1}$ are contained in the projective closure of \emph{every} $Q$-quadric, and $P_d$ is not contained in none of them. 
Such a choice exploits the fact that -- similarly to hyperspheres -- every $Q$-quadric intersects the ideal hyperplane in exactly the same points. 
Using the transitivity automorphism group of projective spaces on points in general position, 
we can freely take the transform of the Veronese curve passing through the points $P_1,\dots,P_d$, see \autoref{lem:interpolateRNC}.
The choice of $P_d$ ensures that this curve is not contained in any $Q$-quadrics. Bézout's theorem then applies and shows that our curve intersects every $Q$-quadric in at most $2d$ points. 
However, exactly $d-1$ of these points were constructed to be in the ideal hyperplane, leaving at most $d+1$ intersection points in the affine space as required.

\subsection{Structure of the paper} 
In \autoref{sec:prelim} we introduce some notation, 
review standard facts from algebraic geometry and number theory. 
In \autoref{sec:Qquadrics}, as a generalisation of hyperspheres, we introduce the notion of a $Q$-quadric for any quadratic form $Q$ over a field and discuss its basic properties.
Finally, in \autoref{sec:constructions}, we prove the main statements in a constructive way first over arbitrary fields (\autoref{sec:ConstrField}), 
then over finite fields (\autoref{sec:ConstrFinite}), 
and finally in finite grids on $\R^d$ (\autoref{sec:ConstrGrid}).

We start every section with a short summary to help the reader with the structure and the goals of the current section.

\subsection{Acknowledgement} 
The author is grateful to Zoltán Lóránt Nagy and Benedek Kovács for the fruitful discussions.

%% file: parts/2-preliminaries.tex
\section{Preliminaries and tools}\label{sec:prelim}
\begin{summary}
    In this section, we fix some notations and collect all the necessary tools for this paper. 
    The first toolkit is algebraic geometry: elementary affine and projective varieties, Bézout's theorem, Warning's theorem about the cardinality of affine varieties over finite fields. 
    The other technical toolkit is classical number theory: quadratic residues, the Prime number theorem, and Dirichlet's theorem on primes in arithmetic progressions.
\end{summary}

\subsection{Notation}

Denote by $\N$, $\Q$ and $\R$ the set of natural, rational and real numbers, respectively. 
For $0<n\in\N$, write $[n]\leteq \{1,2,\dots,n\}$. 

$\field$ denotes an arbitrary field (finite or infinite, algebraically closed or not). 
$\field^\times$ is the multiplicative group of $\field$. 
For a prime power $q$, we denote by $\field_q$ the finite field of order $q$. 
We write $\sqrt{a}\in \field$ to mean that $a\in \field$ is a square, i.e. there exists $x\in\field$ with $x^2=a$. In this paper, it will cause no confusion to write $\sqrt{a}$ for any of $\pm x$. 
 
$\field[X_1,\dots,X_d]$ denotes the polynomial ring over $\field$ in $d$ indeterminates. 
Typically, we have $d\geq 2$. 
For $\delta\in\N$, let $\field[X_1,\dots,X_d]_\delta$ denote the set of homogeneous polynomials $f\in \field[X_1,\dots,X_d]$ of degree $\delta$, 
i.e. every monomial in $f$ with nonzero coefficient has total degree exactly $\delta$. 
Denote the \emph{homogenisation} $f^*\in\field[X_0,\dots,X_d]_\delta$ of $f\in\field[X_1,\dots,X_d]$ of degree $\delta$ by $x_0^\delta f(x_1/x_0,\dots,x_d/x_0)=f^*(x_0,\dots,x_d)$. 
Conversely, the \emph{dehomogenisation} $F_*\in\field[X_1,\dots,X_d]$  of  $F\in\field[X_0,\dots,X_d]_\delta$ is given by $F_*(x_1,\dots,x_d)=F(1,x_1,\dots,x_d)$.

\subsection{Classical algebraic geometry}

For completeness, we recall some standard material that can be found for example in \cite{Harris}.

\paragraph{Affine and projective varieties}
Let $\field$ be an arbitrary field and $d\in\N$.

\begin{definition}[Affine varieties]
    Write $\Aff{\field}{d}\leteq \field^d$ for the $d$-dimensional affine space over $\field$. 
    For a subset $I\subseteq \field[X_1,\dots,X_d]$, 
    define the affine variety by 
    $\Variety{\Aff{\field}{d}}{I}\leteq \{x\in \Aff{\field}{d}:\forall f\in I\quad f(x)=0\}$.  
    For $f\in\field[X_1,\dots,X_d]$ of degree $\delta$, 
    we call $\Variety{\Aff{\field}{d}}{f}\leteq \Variety{\Aff{\field}{d}}{\{f\}}$ a \emph{hypersurface} of degree $\delta$. 
    \emph{Hyperplanes} are hypersurfaces of degree $1$.
\end{definition}

\begin{definition}[Projective varieties]
    Define the \emph{projective space of dimension $d$ over $\field$} to be $\Proj{\field}{d}\leteq (\field^{d+1}\setminus\{0\})/\sim$ where 
    $x\sim y$ if there is $\lambda\in\field^\times$ so that $x=\lambda y$. 
    Denote the $\sim$-equivalence class of $(x_0,\dots,x_d)\in\field^{d+1}\setminus\{0\}$ by $[x_0:\dots:x_d]\in\Proj{\field}{d}$. 
    For a set $I\subseteq \field[X_0,\dots,X_d]$ of homogeneous polynomials, 
    write 
    $\Variety{\Proj{\field}{d}}{I}\leteq \{[x_0:\dots:x_d]\in \Proj{\field}{d}:\forall f\in I\quad f(x_0,\dots,x_1)=0\}$
    for the projective variety. 
    Hypersurfaces and hyperplanes are defined analogously.
\end{definition}

\begin{definition}\label{def:projectivisation}
    For a projective variety $Z=\Variety{\Proj{\field}{d}}{I}$, 
    write $Z_0\leteq \Variety{\Aff{\field}{d}}{\{F_*:F\in I\}}$ for the corresponding affine variety. 
    Write $Z_\infty\leteq \ProjClosure{Z}\cap \Variety{\Proj{\field}{d}}{X_0}=\{[x_0:\dots:x_d]\in \ProjClosure{Z}:x_0=0\}$ for the \emph{ideal point of $V$}.
    Inversely, for an affine variety $Y\leteq \Variety{\Aff{\field}{d}}{I}$, 
    write $\ProjClosure{Y}\leteq \Variety{\Proj{\field}{d}}{\{f^*:f\in I\}}$ for the \emph{projective closure of $V$}. 
    Let $Y_\infty\leteq (\ProjClosure{Y})_\infty$. 
    If the name of the affine variety is $Y_0$ (to emphasise its affine nature), we sometimes implicitly mean $Y\leteq \ProjClosure{Y}$ and consequently $Y_\infty\leteq (Y_0)_\infty$.
    For any field extension $\field\leq \K$, 
    write $Y(\K)\leteq \Variety{\Proj{\K}{d}}{I}$.
\end{definition}

\begin{remark}\label{rem:projectiveClosure}
    The two constructions are inverses of each other: 
$Y=(\ProjClosure{Y})_0$ 
and $\overline{Z_0}=Z$.
Furthermore, there is a bijection $Z_0 \to Z\setminus Z_\infty$ given by $(x_1,\dots,x_d)\mapsto [1:x_1:\dots:x_d]$, whose inverse is $[x_0:\dots:x_d]\mapsto (x_1/x_0,\dots,x_d/x_0)$.     
\end{remark}

\paragraph{Bézout's theorem}
We need to bound the number of intersection points of varieties. 
Denote the degree of a variety $Z$ by $\deg(Z)$, i.e. the number of points of intersection of $Z$ and  a generic $\dim(Z)$-codimensional subspace. 
\begin{theorem}[Bézout's theorem in complementary dimension, {\cite[Theorem~18.4]{Harris}}]
\label{thm:BezoutProj}
   Let $\AlgClosure{}{\field}$ be an algebraically closed field. 
   Let $C, S\subseteq \Proj{\field}{d}$ be irreducible varieties with $\dim(C)+\dim(S)=d$ and $\dim(C\cap S)=0$. 
   Then 
   $\sum_{P\in C\cap S} m(P)=\deg(C)\cdot \deg(S)$
   where $m(P)\in\Z_+$ is the multiplicity of the intersection point $P$.
\end{theorem}

We need the following special case about the intersection of affine curves and hypersurfaces. 
\begin{corollary}
\label{thm:BezoutAff}
   Let $\field$ be an arbitrary field. 
   Let $C_0\subseteq \Aff{\field}{d}$ be an irreducible curve and $S_0\subset \Aff{\field}{d}$ be a hypersurface. 
   If $C_0\not\subseteq S_0$, then \[|C_0\cap S_0|\leq \deg(C_0)\cdot \deg(S_0)-|C_\infty\cap S_\infty|.\]
\end{corollary}
\begin{proof}
    Write $C\leteq \overline{C_0}$ and $S\leteq \overline{S_0}$ for the projective closures, and 
    $\AlgClosure{\field}$ for an algebraic closure of $\field$. 
    Note that $\dim(C(\AlgClosure{\field}))+\dim(S(\AlgClosure{\field}))=d$ and $\dim(C(\AlgClosure{\field})\cap S(\AlgClosure{\field}))=0$ as $C_0\not\subseteq S_0$.
    Then \autoref{def:projectivisation}, \autoref{rem:projectiveClosure} and \autoref{thm:BezoutProj} shows that 
     $|C_0\cap S_0|+|C_\infty\cap S_\infty|
     =|C\cap S|
     \leq | C(\AlgClosure{\field})\cap  S(\AlgClosure{\field})|
    \leq  \deg( C(\AlgClosure{\field}))\cdot \deg( S(\AlgClosure{\field}))
    = \deg(C_0)\cdot \deg(S_0)
    $ and the statement follows.
     \end{proof}

\paragraph{Rational normal curves} 
Our construction uses the following standard properties of the rational normal curves. For completeness, we include a short proof for the ones we actually need. 
For further details, see \cite[\S1,~\S18]{Harris}.
\begin{definition}[Rational normal curve, {\cite[Example~1.14.]{Harris}}]
\label{def:RNC}
    The image of the so called, \emph{Veronese map} \[\Veronese{d}\from \Proj{\field}{1}\to \Proj{\field}{d},\quad [x_0:x_1]\mapsto [x_0^d:x_0^{d-1}x_1:x_0^{d-2}x_1^2:\dots:x_1^d]\] 
    is a degree $d$ irreducible rational curve $\im(\Veronese{d})\subseteq \Proj{\field}{d}$. 
    A curve $C\subseteq \Proj{\field}{d}$ is \emph{rational normal}, if it is projectively equivalent to $\im(\Veronese{d})$, 
    i.e. if there is a linear automorphism $\phi\in\PGL(d+1, \field)$ of $\Proj{\field}{d}$ such that $C=\im (C_d)$.
\end{definition}
\begin{remark}
    This curve is given by $\im(\Veronese{d})=\Variety{\Proj{\field}{d}}{\{X_iX_{j}-X_{i-1}X_{j+1}:1\leq i\leq j\leq d-1\}}$, see \cite[Exercise~5.4]{Harris}.
\end{remark}

    Points $P_0,\dots,P_k\in\Proj{\field}{d}$ are \emph{in general position}, if there is no (linear) subspace $Y$ of dimension $k-1$ with $P_i\in Y$ for all $0\leq i\leq k$.

\begin{lemma}[Standard properties of rational normal curves]
\label{lem:RNC}
    Any rational normal curve $C\subseteq \Proj{\field}{d}$ satisfies the following properties.
    \begin{enumerate}[nosep]
        \item\label{part:Chyperplanes} Arbitrary (pairwise different) points $P_0,\dots,P_{d}\in C$ are in general position, i.e. they generate $\Proj{\field}{d}$.
        Equivalently, we have $|C\cap H|\leq d$ for any hyperplane $H\subset \Proj{\field}{d}$.

        \item The Veronese map $\Veronese{d}\from \Proj{\field}{1}\to \Proj{\field}{d}$ is injective. 
        In particular $|C|=|\field|+1$.        
    \end{enumerate}
\end{lemma}
\begin{proof}
    Let $Q_i=\Veronese{d}([a_i,b_i])$ (for $0\leq i\leq d$)
    for pairwise different $[a_i:b_i]\in\Proj{\field}{1}$. 
    Then $Q_0,\dots,Q_d$ are in general position if and only if 
    the matrix $A=(a_i^jb_i^{d-j})\in \field^{(d+1)\times (d+1)}$ (for $0\leq i,j\leq d$) is invertible.
    To conclude both parts, note that this Vandermonde-like determinant factorises as $\det(A)=\prod_{i\neq k} 
    \det\left(\begin{smallmatrix}
        a_i & b_i \\ a_k & b_k
    \end{smallmatrix}\right)$.    
\end{proof}

\begin{remark}\label{rem:RNCdegChar}
    If $K\subseteq \Proj{\field}{d}$ is an irreducible non-degenerate (not contained in any hyperplane) curve, then $\deg(K)\geq d$ with equality if and only if $K$ is a rational normal curve, \cite[Proposition~18.9]{Harris}. 
    In particular, 
    \autoref{part:Chyperplanes} of \autoref{lem:RNC} actually characterises rational normal curves. 
\end{remark}

Since lower degree curves give better constructions for our main problem, \autoref{rem:RNCdegChar} motivates the study of rational normal curves. 
Fortunately, one can interpolate such a curve to given points rather freely.
\begin{lemma}\label{lem:interpolateRNC}
    If $|\field|\geq d-1$ and $P_1,\dots,P_d\in\Proj{\field}{d}$ are points in general position, 
    then there exists a rational nomral curve $C\subseteq \Proj{\field}{d}$ such that $P_i\in C$ for every $1\leq i\leq d$.
\end{lemma}
\begin{proof}
If $|\field|+1=|\Proj{\field}{1}|\geq d$, then 
we may pick pairwise different points $R_1,\dots,R_d\in \Proj{\field}{1}$. 
Then $Q_i=\Veronese{d}(R_i)\in \im(\Veronese{d})$ (for $1\leq i\leq d$) are pairwise different points in general position by \autoref{lem:RNC}. 
Then there is $\phi\in\PGL(d+1,\field)$ such that $\phi(Q_i)=P_i$ for every $1\leq i\leq d$. 
This means that $C=\phi(\im(\Veronese{d}))$ satisfies the statement.
\end{proof}
\begin{remark}
    Actually, any $d+3$ points of $\Proj{\field}{d}$ in general position determines a \emph{unique} rational normal curve in $\Proj{\field}{d}$, see \cite[Theorem~1.18]{Harris}.
\end{remark}

\paragraph{Affine varieties over finite fields}

The next statement controls the cardinality of varieties over finite fields. 
\begin{theorem}[Warning’s second theorem, {\cite[Satz~3]{Warning}}]
\label{thm:Warning}
    Let $\field$ be a finite field. 
    Let $f_1,\dots,f_m\in\field[X_1,\dots,X_d]$ with $\delta=\sum_{i=1}^m \deg(f_i)$. 
    Then $\Variety{\Aff{\field}{d}}{\{f_1,\dots,f_m\}}$ is either empty or has cardinality at least $|\field|^{d-\delta}$.
\end{theorem}

\begin{corollary}\label{cor:quadricSize}
    If $\field$ is a finite field and 
    $Q\in\field[X_1,\dots,X_d]_2$, 
    then $|\Variety{\Aff{\field}{d}}{Q}|\geq |\field|^{d-2}$
\end{corollary}
\begin{proof}
    If $Q=0$, then the statement is clear. 
    Otherwise note that $(0,\dots,0)\in\Variety{\Aff{\field}{d}}{Q}$ as $Q$ is homogeneous.
    Thus \autoref{thm:Warning} implies the statement as $\delta=\deg(Q)=2$.
\end{proof}

\subsection{Classical number theory}

\paragraph{Quadratic residues}
As a technical tool, we need the following exercise.
\begin{lemma}\label{lem:quadraticResidue}
    Let $\Delta\in\Z\setminus\{0\}$, 
    and $p\in\N$ be a prime with $p\equiv 1\pmod{4|\Delta|}$. 
    Then $\sqrt{\Delta}\in \field_p$.
\end{lemma}
\begin{proof}
    Note that $p\equiv 1\pmod{4}$ by assumption, hence $\sqrt{-1}\in\field_p$ by \cite[Theorem~9.4]{Apostol}.
    Let $\ell$ be a prime divisor of $\Delta$ with odd multiplicity. 
    To finish, we claim that $\sqrt{\ell}\in\field_p$.
    Indeed, if $\ell=2$, then 
    $8\mid 4\Delta \mid p-1$, in which case  $\sqrt{\ell}=\sqrt{2}\in\field_p$ by \cite[Theorem~9.5]{Apostol}. 
    Otherwise, $\ell$ is an odd prime, so the quadratic reciprocity (\cite[Theorem~9.8]{Apostol}) gives 
    $\Legendre{\ell}{p}=\Legendre{p}{\ell}$ keeping in mind that $p\equiv 1\pmod{4}$. 
    But $\ell \mid 4\Delta \mid p-1$, i.e. $p\equiv 1\pmod{\ell}$, thus the Legendre symbol evaluates to $\Legendre{p}{\ell}=\Legendre{1}{\ell}=1$. 
    This means that $\sqrt{\ell}\in \field_p$ as claimed.    
\end{proof}

\paragraph{Distribution of prime numbers} 
We use the following weak form of distribution of primes (in arithmetic  progressions).

\begin{theorem}[Weak form of the Prime number theorem, cf. {\cite[\S 4]{Apostol}}]
\label{thm:Prime}
    For every $\epsilon>0$, 
    there exists $N_\epsilon$, 
    such that for every $n\geq N_\epsilon$, 
    there is a prime number $p$ satisfying $(1-\epsilon)n\leq p\leq n$.
\end{theorem}

\begin{theorem}[Weak form of Dirichlet's theorem, cf. {\cite[\S 7]{Apostol}}]
\label{thm:Dirichlet}
    For every $\epsilon>0$, 
    $1\leq m\in \Z$, and $a\in \N$
    with $\gcd(m,a)=1$, 
    there is $N_{\epsilon,a,m}\in\N$
    such that for every $n\geq N_{\epsilon,a,m}$, there is a prime number $p$ satisfying $(1-\epsilon)n\leq p\leq n$ with $p\equiv a\pmod{m}$.
\end{theorem}

%% file: parts/3-Q-quadrics.tex
\section{\texorpdfstring{$Q$}{Q}-quadrics}\label{sec:Qquadrics}
\begin{summary}
    This section introduces a fundamental notion for this paper. 
    We generalise the notion of hyperspheres to \emph{$Q$-quadrics} where $Q$ is a quadratic form. 
    (Various choices of $Q$ produce families of quadratic hypersurfaces `similar' to each other, e.g. hyperspheres, 
    hyperellipsoids and hyperparaboloids with given axes direction and ratios.) 
    We discuss some properties of $Q$-quadrics that will be used in \autoref{sec:constructions} in the constructions.
\end{summary}

\begin{definition}\label{def:Qquadric}
    Let $Q\in\field[X_1,\dots,X_d]$ be a quadratic form. 
    A subset $S\subseteq \Aff{\field}{d}$ is called a \emph{$Q$-quadric} (over $\field$) 
    if there is $f\in \field[X_1,\dots,X_d]$ of degree at most $1$ 
    such that $S=\Variety{\Aff{\field}{d}}{Q+f}$.
\end{definition}

\begin{remark}\label{rem:sphere}
    Note that for any $\lambda\in\field^\times$, the notion of $Q$-quadrics and $(\lambda Q)$-quadrics coincide.

    Geometrically, the set of all $Q$-quadrics contains all scaled and translated images of $\Variety{\field}{Q}$ and $\Variety{\field}{Q-1}$,  
    but, for example, when the characteristic of $\field$ is $2$, it may contain further elements. 

    The motivating example is $Q=X_1^2+\dots+X_d^2$, in which case every hypersphere is a $Q$-quadric. 
    Other $Q$'s specify certain ellipsoids (with axes pointing to given directions and with given ratios), paraboloids, amongst others. 

    $0$-quadrics are hyperplanes and the whole space, so we usually assume $Q\neq 0$.
\end{remark}

\begin{remark}\label{rem:Q_infty}
    Recall \autoref{def:projectivisation}. 
    For any $Q$-quadric $S$, we have   
    $S_\infty =  \Variety{\Proj{\field}{d}}{\{Q,X_0\}}$. 
    Note that this is independent of $S$, and only depends on $Q$ itself. 
    In other words, all $Q$-quadrics intersect the ideal hyperplane $\Variety{\Proj{\field}{d}}{X_0}$ in exactly the same points. 
    This will be a crucial property for our construction.
\end{remark}

Any $d+1$ points in general position determine a unique $Q$-quadric.

\begin{lemma}\label{lem:uniqueQquadric}
    Let $\field\leq \K$ be fields, 
    $Q\in\field[X_1,\dots,X_d]_2$, 
    and $P_0,\dots,P_d\in\Aff{\field}{d}$ be points in general position.  
    Then there there is a unique $Q$-quadric $S_\K\subseteq\Aff{\K}{d}$ (over the larger field $\K$) containing every $P_i$. 
    In fact, $S_\K=\Variety{\Aff{\K}{d}}{Q+f}$ for a (unique)  $f\in\field[X_1,\dots,X_d]$ of degree at most $1$ defined over the smaller field $\field$.
\end{lemma}
\begin{proof} 
    Write $P_i=(a_{i,1},\dots,a_{i,d})\in\Aff{\field}{d}$, 
    let $a_{i,0}\leteq 1$, 
    and define the matrix $A\leteq (a_{i,j})\in \field^{(d+1)\times (d+1)}$ (for $0\leq i,j\leq d$). 
    Define $b_i\leteq -Q(P_i)\in \field$ for $0\leq i\leq d$, 
    and define the vector $b\leteq (b_0,\dots,b_d)\in \field^{d+1}$. 
    For any $f\leteq \alpha_0+\alpha_1X_1+\dots+\alpha_dX_d \in \K[X_1,\dots,X_d]$, assign the vector
    $x_f\leteq (\alpha_0,\dots,\alpha_d)\in\K^{d+1}$. 
    Now $Ax_f=b$ if and only if $P_i\in \Variety{\K}{Q+f}$ for every $0\leq i \leq d$. 
    But since $P_0,\dots,P_d$ are in general position (over $\field$), the matrix $A$ is invertible (over $\field$), 
    hence there is a unique solution of the previous equation given by $x_f=A^{-1}b\in\field^{d+1}$ as required.
\end{proof}

The fact that any $d$ points of $\Aff{\field}{d}$ lie on a hyperplane together with \autoref{lem:uniqueQquadric} motivates the following definition.
\begin{definition}\label{def:Qgeneric}
    Given a quadratic form $Q\in\field[X_1,\dots,X_d]$, a subset $D\subseteq \Aff{\field}{d}$ is said to be  \emph{$Q$-generic (over $\field$)} if 
    $|D\cap H|\leq d$ for every hyperplane $H\subset \Aff{\field}{d}$, and
    $|D\cap S|\leq d+1$ for every $Q$-quadric $S\subset \Aff{\field}{d}$.
\end{definition}
\begin{remark}
    A $Q$-generic set $D\subseteq \Aff{\field}{d}$ for $Q=X_1^2+\dots+X_d^2$ is a 
    a \noSharp{d} set, i.e. no $d+2$ points of $S$ lie on any sphere and no $d+1$ points lie in a hypersurface.
\end{remark}

The next statement is useful when switching to larger fields.
\begin{lemma}\label{lem:QgenericFieldeEtension}
    Let $Q\in\field[X_1,\dots,X_d]_2$. 
    If $D\subseteq \Aff{\field}{d}$ is $Q$-generic, 
    then it is also $Q$-generic over any larger field 
    $\K\geq \field$.
\end{lemma}
\begin{proof}
    By assumption, any distinct points $P_0,\dots,P_d\in D$ are in general positions over $\field$, but then they are automatically in general position over $\K$ from standard linear algebra.

    For the other condition, assume $|D\cap S_\K|\geq d+1$ for some $Q$-quadric $S_\K\subset \Aff{\K}{d}$ over $\K$. 
    Now any distinct points $P_0,\dots,P_{d}\in D\cap S_\K$ are in general position over $\field$ by \autoref{def:Qgeneric}, so \autoref{lem:uniqueQquadric} implies that $S_\K=\Variety{\Aff{\K}{d}}{Q+f}$ for some $f\in\field[X_1,\dots,X_d]$ defined over the smaller field $\field$. 
    Thus for the $Q$-quadric $S_\field\leteq \Variety{\Aff{\field}{d}}{Q+f}=S_\K\cap \Aff{\field}{d}$, we have $D\cap S_\K = D\cap S_\field$. 
    Hence $|D\cap S_\K|=|D\cap S_\field|\leq d+1$ by \autoref{def:Qgeneric} as required.
\end{proof}

%% file: parts/4-constructions.tex
\section{Constructions and proof of theorems}\label{sec:constructions}
\begin{summary}
    In this section, we prove the main statements of the paper: \autoref{thm:mainShpereGrid}, \autoref{thm:mainQgrid}, \autoref{thm:QgenericoverFinite}, and \autoref{thm:mainSphhereFiniteFields}. 
    All constructions are based on a carefully chosen rational normal curve. We first discuss the general construction over arbitrary fields (\autoref{sec:ConstrField}). 
    Then we show the applicability of this construction for finite fields (\autoref{sec:ConstrFinite}). 
    Finally, from suitable prime fields, we lift back the construction to square grids of the Euclidean space (\autoref{sec:ConstrGrid}).
\end{summary}

\subsection{Over arbitrary fields}\label{sec:ConstrField}
\begin{summary}
    In this section, using projective spaces, we discuss a sufficient condition on the quadratic form $Q$ such that a suitably chosen rational normal curve satisfies the \noQuadric{d}{Q} property. 
    More concretely, this condition is the existence of $d$ independent points from the ideal hyperplane such that every $Q$-quadric passes through exactly $d-1$ of these points. 
    Then we show using Bézout's theorem, that any rational normal curve passing through these points has the desired property, as in the affine space, there can be at most  $2d-(d-1)=d+1$ intersection points left.    
\end{summary}

\begin{definition}\label{def:rich}
    Call a quadratic form $0\neq Q\in\field[X_1,\dots,X_d]_2$ is called \emph{rich}, if there exists a vector space basis $v_1,\dots,v_d$ of $\field^d$ such that $Q(v_1)=\dots=Q(v_{d-1})=0\neq Q(v_d)$. 
\end{definition}

\begin{remark}\label{rem:richProj}
    Richness is equivalent to the existence of general position points $P_1,\dots,P_d\in\Variety{\Proj{\field}{d}}{X_0}$ in the ideal hyperplane 
    such that 
    $P_1,\dots,P_{d-1}\in \Variety{\Proj{\field}{d}}{Q}$ and $P_d\notin \Variety{\Proj{\field}{d}}{Q}$, as we may assign 
    $P_i=[0:v_{i,1}:\dots : v_{i,d}]\in\Variety{\Proj{\field}{d}}{X_0}$ to the vector  $v_i=(v_{i,1},\dots,v_{i,d})\in\field^p$. 
\end{remark}

\begin{remark}
    If $\field$ is algebraically closed, then by Hilbert's Nullstellensatz every $0\neq Q\in\field[X_1,\dots,X_d]$ is rich. 
    However, this is not true in general. 
    For example, in the case 
    $\field=\R$, $d\geq 2$,
    pick $m\geq 1$, and $0<\lambda_i\in\R$,  $L_i\in\R[X_1,\dots,X_d]_1$ arbitrarily for any $i\in\{1,\dots,m\}$.  
    Then the quadratic form $Q=\sum_{i=1}^m \lambda_i L_i^2$ is not rich (as $\Variety{\Aff{\R}{d}}{Q}$ contains a single point).
\end{remark}


We can construct $Q$-generic sets for rich $Q$'s of linear size in $|\field|$.
\begin{proposition}[Constructing $Q$-generic sets]
\label{prop:QgenericOverF}
    Let $\field$ be an arbitrary field, and $0\neq Q\in\field[X_1,\dots,X_d]_2$ be a rich quadratic form where $d\leq |\field|+1$. 
    Then there is a $Q$-generic subset $C_0\subseteq\Aff{\field}{d}$ with $|C_0|=|\field|+1-d$. 

    In fact, $C_0$ may be taken to be the affine part of a suitable rational normal curve over $\field$. 
\end{proposition}
\begin{proof}
    Let $P_1,\dots,P_d\in\Variety{\Proj{\field}{d}}{X_0}$ be the general position points as in \autoref{rem:richProj}. 
    Let $C$ be the degree $d$ rational normal curve through these $P_1,\dots,P_d$ provided by \autoref{lem:interpolateRNC}. 
    We claim that the affine part $C_0\subseteq \Aff{\field}{d}$ satisfies the statement.

    Indeed, 
    to check the first condition of condition of \autoref{def:Qgeneric}, let $H_0\subset \Aff{\field}{d}$ be a hyperplane (which is a hypersurface of degree $1$).
    Then $|C_0\cap H_0|\leq \deg(C_0)\cdot \deg(H_0)=d$ by \autoref{thm:BezoutAff} and \autoref{lem:RNC}.

    To check the other condition, let $S_0\subseteq \Aff{\field}{d}$ be a $Q$-quadric. 
    We have $S_\infty=\Variety{\Proj{\field}{d}}{\{Q,X_0\}}$ by \autoref{rem:Q_infty}, so 
    \autoref{rem:richProj} implies that
    $P_1,\dots,P_{d-1}\in S_\infty$ and $P_d\notin S_\infty$, 
    therefore we have $\overline{C_0}\not\subseteq \overline{S_0}$, i.e. $C_0\not\subseteq S_0$. 
    Next note that $C_\infty=\{P_1,\dots,P_d\}$ by construction and \autoref{lem:interpolateRNC}, 
    hence $C_\infty \cap S_\infty=\{P_1,\dots,P_{d-1}\}$. 
    Thus, \autoref{thm:BezoutAff} gives 
    $|C_0\cap S_0|\leq \deg(C_0)\cdot \deg(S_0)-|C_\infty\cap S_\infty|
    =  2d-|\{P_1,\dots,P_{d-1}\}| = d+1$
    as required.    

    Finally, note that $|\overline{C_0}|=|\Proj{\field}{1}|=|\field|+1$ by \autoref{def:RNC} and $|C_\infty|=|\{P_1,\dots,P_d\}|=d$ by above, so $|C_0|=|\overline{C_0}|-|C_\infty|=|\field|+1-d$ by \autoref{rem:projectiveClosure} as required. 
\end{proof}

\subsection{Over finite fields}\label{sec:ConstrFinite}
\begin{summary}
    In this section, we show that the sufficient condition of \autoref{sec:ConstrField} for finite fields is almost always satisfied. The only exception is when the quadratic form is irreducible and has rank $2$. 
    In every other case, we use the general construction of \autoref{sec:ConstrField} to obtain \noQuadric{d}{Q} and \noSharp{d} sets of linear size over finite fields, thereby proving  \autoref{thm:QgenericoverFinite} and \autoref{thm:mainSphhereFiniteFields}.
\end{summary}

If we consider only hyperspheres, i.e. $Q=X_1^2+\dots+X_d^2$, then it is easy to show that if $\sqrt{-1}\in\field$, then $Q$ is rich by considering the cyclic permutations of the point $(1,\sqrt{-1},0,\dots,0)\in\Variety{\Aff{\field}{d}}{Q}$, cf. CITE. 
Classifying rich quadratic forms over finite fields is not much harder.
\begin{lemma}\label{lem:richOverFinite}
    If $\field$ is finite, then every quadratic form 
    $0\neq Q\in\field[X_1,\dots,X_d]_2$ is either rich or is irreducible of rank $2$.
\end{lemma}
\begin{remark}\label{rem:rank2irreducibleQuadratics}
    Irreducible rank $2$ quadratic forms can be given explicitly by $Q=\lambda_{1,1}L_1^2+\lambda_{1,2}L_1L_2+\lambda_{2,2} L_2^2$ for $\field$-linearly independent $L_1,L_2\in\field[X_1,\dots,X_d]_1$ and  $\lambda_{i,j}\in\field$ 
    such that $\sqrt{\Delta}\notin\field$ for the the discriminant $\Delta=\lambda_{1,2}^2-4\lambda_{1,1}\lambda_{2,2}\in\field$.
\end{remark}

\begin{proof}
Let $V_Q\leteq \Variety{\Aff{\field}{d}}{Q}$. 
By \autoref{cor:quadricSize}, we have $|V_Q|\geq |\field|^{d-2}$. 
Considering the cardinalities show that $\dim(\Span_\field(V_Q))\geq d-2$ with equality if and only if $V_Q\subseteq \Aff{\field}{d}$ is a $(d-2)$-dimensional linear subpsace.  
Thus $Q$ fails to be rich if and only if $V_Q$ is a $2$-codimensional subspace. 
In this case, after an appropriate change of base of $\field[X_1,\dots,X_d]$ from $\{X_1,\dots,X_d\}$ to degree $1$ homogeneous polynomials $\{L_1,\dots,L_d\}$, 
we may assume that this subspace is given by $L_1=L_2=0$. 
Denote by $Q_1\in\field[L_1,\dots,L_d]$ the resulting quadratic form in the new basis. 
Then 
i.e. that $Q_1(y_1,\dots,y_d)=0$ if and only if $y_1=y_2=0$. 
Write \[Q_1=\sum_{1\leq i\leq j\leq d} \lambda_{i,j}L_iL_j\] 
for suitable $\lambda_{i,j}\in\field$.
Let $e_i\in\Aff{\field}{d}$ where the $i$th entry is $1$, all others are $0$ (in the new basis).
Note that 
$Q_1(e_i+\mu e_j)=\lambda_{i,i}+\mu\lambda_{i,j}+\mu^2\lambda_{j,j}$  for every $i<j$ and $\mu\in\field$. 
Using this with $3\leq i<j$ and $\lambda\in\{0,1\}$, we see that $\lambda_{i,j}=0$ for every $3\leq i\leq j$. 
Then for $i\in\{1,2\}$, $3\leq j$ and arbitrary $\mu\in\field$, we have 
$0\neq Q_1(e_i,\mu e_j)=\lambda_{i,i}+\mu\lambda_{i,j}$ which forces $\lambda_{i,j}=0$. 
Now $Q_1 = \lambda_{1,1}L_1^2+\lambda_{1,2}L_1L_2+\lambda_{2,2} L_2^2$.
Finally $0\neq Q_1(e_1+\mu e_2)=\lambda_{1,1}+\mu\lambda_{1,2}+\mu^2\lambda_{2,2}$ holds for every $\mu\in\field$, i.e. the discriminant  $\lambda_{1,2}^2-4\lambda_{1,1}\lambda_{2,2}\in\field$ (of this quadratic expression in $\mu$) must not be a square in $\field$. 
Going back to the original basis gives the statement.
\end{proof}

We are ready to prove the main statements about finite fields.

\begin{proof}[Proof of \autoref{thm:QgenericoverFinite}]
    $Q$ is rich by \autoref{lem:richOverFinite}, so \autoref{prop:QgenericOverF} gives the statement.
\end{proof}

\begin{proof}[Proof of \autoref{thm:mainSphhereFiniteFields}]
    Let $Q=X_1^2+\dots+X_d^2$. Note that $Q$-quadrics are the hyperspheres in $\Aff{\field}{d}$, see \autoref{rem:sphere}. 
    Since $\sqrt{-1}\notin \field$ if and only if $|\field|\not\equiv 3\pmod{4}$, 
    \autoref{rem:rank2irreducibleQuadratics} and \autoref{thm:QgenericoverFinite} give the statement. 
\end{proof}

\subsection{In finite square grids of \texorpdfstring{$\R^d$}{R\^n}}\label{sec:ConstrGrid}
\begin{summary}
    In this final section, 
    we prove the remaining main statements of the paper (\autoref{thm:mainShpereGrid}, \autoref{thm:mainQgrid}). 
    We consider a large prime $p\leq n$ (satisfying some technical conditions) and work in a $d$-dimensional grid of size $p^d$. 
    We show that the solutions for $\Aff{\field_p}{d}$ can be lifted to the Euclidean plane, such that every hyperplane and $Q$-quadric with rational or real coefficients intersect our construction in more points than prescribed. 
\end{summary}

\begin{definition}
\label{def:piRho}
    Let $p$ be a prime. 
    Write $\field_p \leteq  \Z/(p)$ and let $\pi_p\from \Z\to \field_p$ be the natural projection. 
    Write $\pi_p^d\from \Z^d\to \Aff{\field_p}{d}$ for the induced map. 
    Define a map (of sets) $\rho_p\from \Q[X_1,\dots,X_d]\to \field_p[X_1,\dots,X_d]$ as follows. 
    Let $\rho_p(0)\leteq 0$. 
    For $0\neq F$, let $r\in\Q^\times$ be (the unique rational) so that $rF$ has integer coefficients whose greatest common divisor is $1$. 
    We obtain $\rho_p(F)\neq 0$ from $rF$ by replacing the coefficients by their image under $\pi_p$.
\end{definition}

\begin{lemma}\label{lem:lift}
    Let $p$ be a prime number. 
    Assume that $D_p\subseteq \Aff{\field_p}{d}$ is a $Q_p$-generic for some $Q_p\in\field_p[X_1,\dots,X_d]_2$.
    Define $D\subseteq \{1,2,\dots,p\}^d$ by $\pi_p^d(D)=D_p$ and 
    pick $Q\in\Q[X_1,\dots,X_d]$ so that $\rho_p(Q)=Q_p$. 
    Then $D\subset \R^d$ is $Q$-generic over $\R$     with $|D|=|D_p|$.
\end{lemma}
\begin{proof}
    By \autoref{lem:QgenericFieldeEtension}, it is enough to show that $D$ is $Q$-generic over $\Q$. 
    To prove the latter one, we claim that for any $F\in\Q[X_1,\dots,X_d]$, we have 
   $|\{x\in D:F(x)=0\}|\leq |\{y\in D_p:\rho_p(F)(y)=0\}|$. 
   Indeed, if $F(x)=0\in \Q^d$, then $rF(x)=0\in\Z^d$, so 
   $\rho_p(F)(y)=0\in\Aff{\field}{d}$ 
   for $y\leteq \pi_p^d(x)\in D_p$ which is uniquely determined by $x$.

    Let $H=\Variety{\Aff{\Q}{d}}{f}$ be an arbitrary hyperplane (in $\Aff{\Q}{d}$) given by the degree $1$ polynomial $f\in\Q[X_1,\dots,X_d]$. 
    If $\deg(\rho_p(f))=1$, 
    then $H_p\leteq \Variety{\Aff{\field_p}{d}}{\rho_p(f)}$ is also hyperplane (in $\Aff{\field}{d}$), 
    so the claim shows that $|D\cap H|\leq |D_p\cap H_p|\leq d$ using the assumption of $D_p$ being $Q_p$-generic. 
    Otherwise, by \autoref{def:piRho}, $\rho_p(f)$ is a non-zero constant polynomial, hence $f$ has no solution on $\Z^d$, thus we have $|D\cap H|=0\leq d$ in this case as well. 

    Finally, $S=\Variety{\Aff{\Q}{d}}{Q+f}$ be an arbitrary $Q$-quadric (over $\Q$) given by the polynomial $f\in\Q[X_1,\dots,X_d]$ of degree at most $1$. 
    Let $F\leteq Q+f$ and define $F_p\leteq \rho_p(F)$. 
    If $\deg(\rho_p(Q+f))=2$, then 
    $S_p\leteq \Variety{\Aff{\field_p}{d}}{\rho_p(Q+f)}$ is a $Q_p$-quadric (as $\rho_p(Q+f)=\rho_p(Q)+\rho_p(f)$), 
    so $|D\cap S|\leq |D_p\cap S_p|\leq d+1$ as above using that $D_p$ is $Q_p$-generic.
    If $\deg(\rho_p(Q+f))\leq 1$, then $S_p=\Variety{\Aff{\field_p}{d}}{\rho_p(f)}$, so we have 
    $|D\cap S|\leq d\leq d+1$ as above.    
\end{proof}

\begin{proposition}\label{thm:mainEps}
    Fix a quadratic form $0\neq Q\in  \Q[X_1,\dots,X_d]_2$. 
    Then for every $0<\epsilon\leq 1$ there exists $N=N(\epsilon,Q, d)\in\N$ such that whenever 
    $n\geq N$, 
    there is $Q$-generic set $D\subseteq \{1,2,\dots,n\}^d$ over $\Q$ of size $|D|\geq (1-\epsilon)n$.
\end{proposition}
\begin{proof}
    First, consider the case when $Q$ is not irreducible of rank $2$, see \autoref{rem:rank2irreducibleQuadratics}. 
    By \autoref{rem:sphere}, we may assume that $Q=\sum_{1\leq i\leq j\leq d} \lambda_{i,j}X_iX_j\in\Z[X_1,\dots,X_d]_2$. 
    Let $C_Q\leteq \{2|\lambda_{i,j}|+1:1\leq i\leq j\leq d\}$. 
    We claim that $N\leteq \max\{N_{\epsilon/2},C_Q,2(d-1)/\epsilon\}$ satisfies the statement where $N_{\epsilon/2}$ is from \autoref{thm:Prime}. 

    Indeed, let $n\geq N$. 
    By \autoref{thm:Prime}, pick a prime number $p$ with $(1-\epsilon/2)n\leq p \leq n$. 
    As $n\geq C_Q$, 
    the quadratic form $Q_p\leteq \rho_p(Q)\in\field_p[X_1,\dots,X_d]_2$ is not irreducible of rank $2$, because $Q$ is not. 
    Since $n\geq 2(d-1)/\epsilon$, we see that 
    $|\field_p|+1-d\geq (1-\epsilon/2)n+ (\epsilon/2 ) n = (1-\epsilon)\geq 0$, 
    hence \autoref{thm:QgenericoverFinite} produces a $Q_p$-generic set $D_p\subseteq \Aff{\field_p}{d}$ of size $|D_p|=p+1-d$. 
    Thus \autoref{lem:lift} produces a $Q$-generic set $D\subseteq \{1,\dots,p\}^d$ over $\R$ of size $|D|=|D_p|=p+1-d\geq (1-\epsilon)n$ as required.

    The case when $Q$ is irreducible of rank $2$ is analogous to the following slight variation. Write $\Delta$ for the discriminant of $Q$ from \autoref{rem:rank2irreducibleQuadratics}, and this time use \autoref{thm:Dirichlet} to get a suitable prime $p$ satisfying $p\equiv 1 \pmod{4|\Delta|}$ with $(1-\epsilon/2)n\leq p \leq n$ $p\leq (1-\epsilon=2)n$. 
    Now \autoref{lem:quadraticResidue} and \autoref{rem:rank2irreducibleQuadratics} shows 
    that the quadratic from $Q_p\in\field_p[X_1,\dots,X_d]$ is of degree $2$, but is reducible. 
    Thus \autoref{thm:QgenericoverFinite} is applicable and finishes the proof as above.  
\end{proof}

With this, the proof of the main statements are complete.

\begin{proof}[Proof of \autoref{thm:mainQgrid}]
    This is a reformulation of \autoref{thm:mainEps}.
\end{proof}

\begin{proof}[Proof of \autoref{thm:mainShpereGrid}]
    This follows from \autoref{thm:mainQgrid} in the case $Q=X_1^2+\dots+X_d^2$.
\end{proof}

%% file: parts/5-open.tex
\section{Open problems}

It is a natural open problem to find the exact value of $f([n]^d,\R^d)$ from \autoref{p:noSharp}, 
or more generally, the exact value of $f_Q([n]^d,\R^d)$ from \autoref{p:noQuadric}. 
In fact, any improvement in the bounds of \autoref{thm:mainShpereGrid} and \autoref{thm:mainQgrid} would be interesting and quite possibly requires new ideas. 
Similar questions arise for the finite field analogue (\autoref{thm:QgenericoverFinite}), where improving the bounds to $f_Q(\field^d,\field^d)$ may be easier.